\documentclass[11pt,a4paper]{amsart}

\usepackage[utf8]{inputenc}
\usepackage[T1]{fontenc}
\usepackage[american]{babel}
\usepackage{lmodern}
\usepackage{amsmath,amssymb,amsthm}
\usepackage[pdfencoding=auto]{hyperref}
\usepackage[left=3.5cm,right=3.5cm]{geometry}

\newcommand{\IC}{\mathbb{C}}

\newcommand{\Z}{{\mathbb Z}}

\newcommand{\lra}{\longrightarrow}
\newcommand{\abs}[1]{\lvert#1\rvert}
\newcommand{\norm}[1]{\lVert#1\rVert}

\newcommand{\Hm}[1]{\leavevmode{\marginpar{\tiny%
$\hbox to 0mm{\hspace*{-0.5mm}$\leftarrow$\hss}%
\vcenter{\vrule depth 0.1mm height 0.1mm width \the\marginparwidth}%
\hbox to 0mm{\hss$\rightarrow$\hspace*{-0.5mm}}$\\\relax\raggedright
#1}}}

\newtheorem{theorem}{Theorem}[section]
\newtheorem{corollary}[theorem]{Corollary}
\newtheorem{proposition}[theorem]{Proposition}
\newtheorem{lemma}[theorem]{Lemma}

\theoremstyle{definition}
\newtheorem{definition}[theorem]{Definition}

\newtheorem{example}[theorem]{Example}

\newtheorem{remark}[theorem]{Remark}

\title[Sobolev inequalities and eigenvalue growth]{Sobolev-Type Inequalities and Eigenvalue Growth on Graphs with Finite Measure}

\author{Bobo Hua}
\address[B. Hua]{School of Mathematical Sciences, LMNS, Fudan University, Shanghai 200433, China.}
\email{bobohua@fudan.edu.cn}

\author{Matthias Keller}
\address[M. Keller]{Institut f\"ur Mathematik, Universit\"at Potsdam,
 14476  Potsdam, Germany}
\email{{matthias.keller@uni-potsdam.de}}
\author{Michael Schwarz}
\address[M. Schwarz]{Institut f\"ur Mathematik, Universit\"at Potsdam, 14476  Potsdam, Germany}
\email{mschwarz@math.uni-potsdam.de}
\author{Melchior Wirth}
\address[M. Wirth]{Institute of Mathematics, Department of Mathematics and Computer Science, Friedrich Schiller University Jena, 07737 Jena, Germany}
\email{melchior.wirth@uni-jena.de}

\begin{document}
\begin{abstract}
In this note we study the eigenvalue growth of infinite graphs with discrete spectrum. We assume that the corresponding Dirichlet forms satisfy certain Sobolev-type inequalities and that the total measure is finite. In this sense, the associated operators on these graphs display similarities to elliptic operators on bounded domains in the continuum. Specifically, we prove lower bounds on the eigenvalue growth and show by examples that corresponding upper bounds can not be established.
\end{abstract}

\maketitle

\section*{Introduction}
{The spectral theory of the Laplacian on Euclidean domains and Riemannian manifolds has been intensively studied in the literature, see e.g. \cite{CouHil53,Chavel84,SY94}}. In these situations, discrete spectrum for the Laplacian on Riemannian manifolds is known to occur for two reasons. Vaguely one could say the manifold is either ``very bounded'' or ``very unbounded''. Precisely, ``very bounded'' means compact or pre-compact with suitable boundary conditions and the discreteness of the spectrum of the Laplacian in this case is mathematical folklore. On the other hand, ``very unbounded'' can be expressed as uniform unbounded sectional curvature growth as was shown by Donnelly/Li \cite{DL}. 

{Our interest lies in the analogous question for graph Laplacians. As compact graphs are finite and the corresponding graph Laplacian is a finite dimensional operator, so the discreteness of the spectrum is of course trivial in this case.} However, in recent years the investigation of discreteness of spectrum on infinite graphs emerged. The ``very unbounded'' case can be made precise via notions of curvatures \cite{Ke10,Woj08} or more generally unbounded vertex degree growth under the presence of some type of isoperimetric estimate \cite{BGK15}. More recently, also non-trivial ``very bounded'' cases in the sense of pre-compactness of the graphs were investigated, \cite{GHKLW15,KLSW17}. These notions can be equivalently described by some type of Sobolev embedding.

{After having established discreteness of the spectrum, the next question concerns the asymptotics of the eigenvalue growth.} On {manifolds} the asymptotics of the eigenvalues are governed by Weyl's law or Weyl's asymptotics, which is again mathematical folklore in the compact case and also holds to some degree in the non-compact case (see \cite{Cheng75,KLP}). In contrast, for graphs every asymptotics of eigenvalues is possible in the ``very unbounded'' case, see \cite{BGK15}. Furthermore, in \cite{GHKLW15,KLSW17} lower bounds on the eigenvalue asymptotics were given in the ``very bounded'' case.

The purpose of this note is to explore the asymptotics of eigenvalues for graphs, {or, more generally, discrete measure spaces}, under the presence of a Sobolev-type inequality and finite measure. {As main results, we will prove lower bounds} on the eigenvalues $ \lambda_{k} $ of the form
\begin{align*}
\lambda_{k} \ge C k^{1-\frac{2}{p}} -c,\qquad k\in\mathbb{N},
\end{align*}
for certain explicit constants $ C$ and $c $, and $ p $ being the exponent in the Sobolev-type inequality, see Theorem~\ref{spectral_bound_dist_gen},~\ref{spectral_bound_lin} and~\ref{thm:pSobolev}. In particular, this extends the results of \cite{GHKLW15,KLSW17}.

{In the continuum,  the proof of the Weyl law relies on an analysis of the short-time behavior of the heat kernel, which is quite different for graphs (see \cite{KLMST16}). Therefore this proof cannot be adapted to our setting. We adopt two strategies to prove our main results. The first one relies on elliptic methods. We use the $\ell^\infty$-Sobolev inequality, finite measure and the discreteness of the space to prove the lower bound in Theorem~\ref{spectral_bound_lin}. The other strategy uses parabolic methods. Following  Cheng and Li \cite{CL81}, we obtain the heat kernel on-diagonal upper bound via the $\ell^p$-Sobolev inequality, $p\in (2,\infty]$, and derive the lower bound for the eigenvalue growth by considering the heat trace, see Theorem~\ref{thm:pSobolev}.}

In contrast to the continuum case we show that in the discrete case one can not establish corresponding upper bounds. Specifically, one can achieve any eigenvalue growth by manipulating the measure, see Proposition~\ref{no_upper_bound}.

The paper is structured as follows. In the next section we set the stage by introducing the basic notions and presenting the main results. 
These results can be separated into the cases $ p=\infty $ and $p\in(2,\infty)$. The first case,  $ p=\infty $, is proven in Section~\ref{sec:p=infty}, 
and the second case, $ p\in(2,\infty)$, is proven in Section~\ref{sec:p>2}. Finally, in Section~\ref{sec:ex}, we discuss the classes of canonically compactifiable and uniformly transient graphs. In this section we also discuss sharpness of our estimates.

\scriptsize
\noindent\textbf{Acknowledgments: }
The research of M.K. was  supported by the priority program SPP2026 of the German Research Foundation (DFG). M.W. gratefully acknowledges financial support by the German Academic Scholarship Foundation (Stu\-dien\-stif\-tung
des deut\-schen Vol\-kes) and by the German Research Foundation (DFG) via RTG 1523/2. 
\normalsize

\section{Dirichlet Forms on Discrete Measure Spaces and Main Results}\label{sec:pre} 
Let $X$ be an at most countable set. {Via additivity a function $m\colon X\lra (0,\infty)$ is extended  to a measure on $X$.  We call the pair $(X,m)$ a \emph{discrete measure space} and write $\ell^p(X,m)$ for the corresponding $\ell^p$-spaces and $ \norm{\cdot}_{p} $ for their norms, $ p\in[1,\infty] $.} For $ p=2 $, we denote the scalar product by $ \langle\cdot,\cdot\rangle $.

A map $C\colon\IC\lra\IC$ with $C(0) = 0$ and {$\lvert C(s)-C(t)\rvert \leq\lvert s-t\rvert$ for $s,t\in\IC$} is called a normal contraction. A densely defined, closed, real quadratic form $ Q $ on $\ell^2(X,m)$ with domain $ D(Q) $ that satisfies $C\circ u\in D(Q)$ and $Q(C\circ u) \leq Q(u)$ for all $u\in\ell^2(X,m)$ and
all normal contractions $C$ is called a \emph{Dirichlet form}. 

In this paper we study the spectral theory of the self-adjoint positive generator $ L $ of  $ Q $ which is uniquely determined by $ Q(f)=\|L^{1/2}f\|_{2}^{2} $, $ f\in D(Q) $.

\begin{example}[Dirichlet forms induced by graphs]\label{ex:Q}
 Given an at most countable set $X$, a pair $(b,c)$ consisting of a symmetric function $b\colon X\times X\lra[0,\infty)$ that vanishes on the diagonal and satisfies
$$\sum_{y\in X} b(x,y)<\infty$$ for all $x\in X$,
and a function $c\colon X\lra [0,\infty)$ is called a graph on $X$. We equip the set $X$ with a measure $m$ such that $(X,m)$ is a discrete measure space. Then, there are two natural 
Dirichlet forms $Q^{(D)}$ and $Q^{(N)}$ on $\ell^2(X,m)$ that act on functions $f$ in their respective domains as \begin{align}\tag{$ Q $}\label{eq:Q}
f\mapsto\frac 1 2\sum_{x,y\in X}b(x,y)\abs{f(x)-f(y)}^2+\sum_{x\in X} c(x)\abs{f(x)}^2.
\end{align}
For more details we refer to Section~\ref{sec:ex}.
\end{example}

\begin{definition}[Sobolev inequalities]\label{def:Sobolev}
Let $(X,m)$ be a discrete measure space and $Q$ be a Dirichlet form on $\ell^2(X,m)$. For any $p\in (2,\infty],$ $\alpha\geq0,$ and $o\in X,$ we say that $Q$ satisfies a 
$\ell^p_{\alpha,o}$-\emph{Sobolev inequality} if there is a constant $S$ such that $$\|u\|_p^2\leq S(Q(u)+\alpha|u(o)|^2)$$ for all $u$ in the form domain of $Q.$ 
The minimal constant $S$ satisfying the above inequality is called the $\ell^p_{\alpha,o}$-\emph{Sobolev constant}, denoted by $S_{p,\alpha,o}(Q).$ 

In the case $\alpha=0,$ the $\ell^p_{\alpha,o}$-Sobolev inequality reduces to
\begin{equation}
\tag{$S_p$}\label{def:sb}\|u\|_p^2\leq SQ(u),
\end{equation}
which we call the $\ell^p$-\emph{Sobolev inequality}. The minimal constant $S$ in $(S_p)$ is called $\ell^p$-Sobolev constant and denoted by $S_p(Q)$.
\end{definition}

Next, we discuss some examples of graphs satisfying Sobolev inequalities. 
\begin{example}[The Euclidean lattice]
We consider the lattice $\Z^d,$ $d\geq 3$, as the graph with standard weights, that is, $c\equiv 0$ and $b(x,y)=1$ if the Euclidean distance between $x$ and $y$ is $1$, and $b(x,y)=0$ otherwise. Furthermore, we equip $ \Z^{d} $ with the counting measure.
Then,  the quadratic form acting on  $\ell^2(X,m)$ as \eqref{eq:Q} from Example~\ref{ex:Q}  satisfies \eqref{def:sb} for $p=\frac{2d}{d-2}$, see \cite[Theorem~3.6]{HuaMug15}.
\end{example}
This example is only a special case of the following class of examples:
\begin{example}[Graphs satisfying isoperimetric inequalities]
Let $(b,0)$ be a graph over a discrete measure space $(X,m)$  such that the  quadratic form acting as \eqref{eq:Q} from Example~\ref{ex:Q} is bounded on $\ell^2(X,m)$. If the graph satisfies a $d$-dimensional isoperimetric inequality for $d>2,$ see \cite[Chapter~I.4]{Woess00} for the definition, then the form satisfies \eqref{def:sb} for $p=\frac{2d}{d-2}$ \cite[Theorem~3.6]{HuaMug15}. 
\end{example}
\begin{example}[Graphs with the normalizing measure]
Let a graph $(b,0)$ on an at most countable discrete set $X$ be equipped with the normalizing measure $n$ given by $n(x):=\sum_{y\in X}b(x,y)$.
Suppose the form acting as \eqref{eq:Q} satisfies \eqref{def:sb} for $p\in(2,\infty]$ and $\inf_{x\in X}n(x)>0.$ 
Then, for any graph $(b,0)$ over $ (X,m) $ with finite  measure, i.e., $m(X)<\infty$, the corresponding form acting as \eqref{eq:Q} satisfies \eqref{def:sb} (which can be easily seen since there exists $ c>0 $ such that $ n\ge c m $ by the uniform positivity of  $ n $ and the finiteness and, thus, boundedness of $ m $). 
This yields many examples with finite  measure satisfying \eqref{def:sb}.
\end{example}

Keller, Lenz, Schmidt and Wojciechowski \cite{KLSW17} introduced the notion of \emph{uniformly transient graphs} and  proved that a  graph $(b,0)$ over $ (X,m) $ is 
uniformly transient if and only if independently of the finite measure $m$ the form $Q^{(D)}$ (see Section~\ref{sec:ex}) satisfies an $\ell^\infty$-Sobolev inequality,
i.e., $(S_\infty)$ holds. 
Using the Grothendieck factorization principle, they  proved that the spectrum of the Dirichlet Laplacian of a uniformly transient graph 
with finite  measure is discrete. Moreover, c.f. \cite[Theorem~7.3]{KLSW17}, they proved a lower bound for the growth of the eigenvalues of the Dirichlet Laplacian. 

Using basically the same proof as in \cite{KLSW17} we get the same lower bound in our more general situation.
\begin{theorem}\label{spectral_bound_dist_gen}
Let $(X,m)$ be a discrete measure space with finite  measure $ m(X)<\infty $.  Let $L$ be the generator of a Dirichlet form $Q$ that satisfies a 
$\ell^\infty_{\alpha,o}$-Sobolev inequality for some $\alpha\geq 0,o\in X$. Then, the spectrum of the generator $L$ is purely discrete and for every   enumeration  $(x_k)_{k\in\mathbb{N}}$ of $X$ with $x_1=o$ and  every $k\geq 1$ we have
\begin{equation}\label{eq:keller}
\lambda_{k}\geq \frac 1{S_{\infty,\alpha,o}(Q) m(X\setminus\{x_1,\dots,x_{k-1}\})},
\end{equation}
where $\lambda_k$ denotes the $k$-th eigenvalue of $L$ counted with multiplicity. 
\end{theorem}

Note that this result depends on the choice of the enumeration of the set $X$.

Using the same setting as in the theorem above, we will furthermore prove an eigenvalue estimate which is independent of the enumeration. 
After the theorem, we will discuss that both are complimentary results. 

\begin{theorem}\label{spectral_bound_lin}
Let $(X,m)$ be a discrete measure space with finite  measure $ m(X)<\infty $. Let $L$ be the generator of a Dirichlet form $Q$  that  satisfies a 
$\ell_{\alpha,o}^\infty$-Sobolev inequality for some $\alpha\geq 0,o\in X$. Then, the spectrum of the generator $L$ is purely discrete and for all $k\geq 1$
\begin{equation}\label{eq:mainestimate}
\lambda_k\geq \frac{k}{S_{\infty,\alpha,o}(Q) m(X)}-\frac{\alpha}{m(o)},
\end{equation}
where $\lambda_k$ denotes the $k$-th eigenvalue of $L$  counted with multiplicity.
\end{theorem}
 
\begin{remark}
Whether the bound in Theorem \ref{spectral_bound_lin} or in Theorem \ref{spectral_bound_dist_gen} is better for a given measure $m$ strongly depends on the 
distribution of $m$: If $m$ decays slowly, Theorem \ref{spectral_bound_dist_gen} does not imply linear growth of the eigenvalues.
Consider for example a uniformly transient graph, (see Section~\ref{section:uniform}), $(b,0)$ over $(\{x_n\}_{n=1}^\infty,m)$ with $m(x_n)=\frac{1}{n^a}$ and $a>1.$ 
Then estimate \eqref{eq:mainestimate} 
is better than \eqref{eq:keller} for $1<a<2,$ while the estimate \eqref{eq:keller} is better for $a>2.$
In particular, our estimate \eqref{eq:mainestimate} is complementary to the estimate \eqref{eq:keller_original} proven in \cite{KLSW17}. 
\end{remark}
\begin{remark}
In Lemma~\ref{no_upper_bound} we will show that by Theorem~\ref{spectral_bound_dist_gen}, there is no general upper bound on the growth of the eigenvalues.
\end{remark}

\begin{remark} The estimate on the eigenvalue growth of  graphs satisfying $ (S_{\infty}) $ is optimal in the exponent of $k$: For any $\varepsilon>0,$ there is a uniformly transient graph 
such that the eigenvalues of Dirichlet Laplacian $L^{(D)}$ satisfy $\lambda_k=o(k^{1+\varepsilon})$ as $k\to\infty,$ see Proposition~\ref{prop:sharpexample}.
\end{remark}

\begin{remark}
If $\alpha=0$ and $c=0$, the coefficient $(S_\infty(Q) m(X))^{-1}$ is optimal in the following sense: As proven in \cite[Theorem 5.4]{LSS18}, 
for every uniformly transient graph $(b,0)$ and $\varepsilon,M>0$ there exists a measure $m$ of full support with  mass $M=m(X)$ such that \begin{align*}
\lambda_1\leq \frac 1{S_{\infty}(Q) M}+\varepsilon.
\end{align*}
\end{remark}

\begin{remark} There is no uniform upper bound for the eigenvalue growth of  graphs satisfying  $ (S_{\infty}) $.  For any uniformly transient graph $(b,0)$ over $X$ (see Section~\ref{section:uniform} for definition) and any sequence $(a_k)$ of positive numbers, there exists a probability measure ${m}$ such that the eigenvalues $\lambda_k$ of $L^{(D)}$ satisfy
$
\frac{\lambda_k}{a_k}\to \infty, \mathrm{as}\ k\to\infty,
$ see Proposition~\ref{no_upper_bound}.

\end{remark}

In Riemannian geometry, Cheng and Li \cite{CL81} used the heat kernel method to estimate the eigenvalues of the Dirichlet Laplacian on a compact manifold with 
boundary by the Sobolev constant. Following their geometric argument, we estimate the eigenvalues of the generator of a Dirichlet form that 
satisfies \eqref{def:sb} for some $p\in (2,\infty]$. 

\begin{theorem}\label{thm:pSobolev} Let $(X,m)$ be a discrete measure space with finite  measure $ m(X)<\infty $. Let $L$ be the generator of  a Dirichlet form  $Q$ satisfying   \eqref{def:sb} for $p\in (2,\infty].$ 
 Then, the semigroup $(e^{-tL})_{t>0}$ is trace class. In particular, the spectrum of the generator $L$ of $Q$ is discrete and for any $k\geq 1,$
$$\frac1k\sum_{i=1}^k\lambda_i\geq \frac{1}{eS_p(Q)}\left(\frac{k}{m(X)}\right)^{1-\frac{2}{p}}.$$ In particular,
$$\lambda_k\geq \frac{1}{eS_p(Q)}\left(\frac{k}{m(X)}\right)^{1-\frac{2}{p}},$$ where $\lambda_k$ denotes the $k$-th eigenvalue of $L$ counted with multiplicity.
\end{theorem}
 For finite graphs, a similar result to the theorem above is proven in \cite{CY95}. They use the geometric arguments of \cite{CL81}, as well. 
However, for infinite graphs this yields a new class of graphs with purely discrete spectrum.

\section{\texorpdfstring{Eigenvalue estimates for $p=\infty$}{Eigenvalue estimates for p=∞}}\label{sec:p=infty}

In this section we prove Theorem~\ref{spectral_bound_dist_gen} and Theorem~\ref{spectral_bound_lin}.  So, we are interested in Dirichlet forms that satisfy the Sobolev inequality in Definition~\ref{def:Sobolev} with $ p=\infty $. 
We first show that an  $\ell^\infty_{\alpha,o}$-Sobolev inequality implies that the generator of the form has discrete spectrum.

\begin{proposition}\label{spectrum_discrete}
Let $(X,m)$ be a discrete measure space with finite  measure $ m(X)<\infty $. Let $L$ be the generator of a Dirichlet form $Q$  that satisfies a
$\ell^\infty_{\alpha,o}$-Sobolev inequality for some $\alpha\geq 0$, $o\in X$. Then, the semigroup $(e^{-tL})_{t>0}$ is trace class. In particular, the spectrum of $L$ is purely discrete.
\end{proposition}
We follow the proof given in \cite[Theorem 5.1]{GHKLW15}, where this result is proven for canonically compactifiable graphs, see Section~\ref{section:canon}.

\begin{proof}
For all $t>0$ we have
$$e^{-tL}\ell^2(X,m)\subseteq D(L)\subseteq D(Q)\subseteq \ell^\infty(X),$$
where the last inclusion is a consequence of the
$\ell^\infty_{\alpha,o}$-Sobolev inequality. An application of the closed graph theorem yields that $e^{-tL}$ is a bounded operator from $\ell^2(X,m)$ to $\ell^\infty(X)$.
Moreover, by the finiteness of the measure, we have a continuous embedding $j$ of $\ell^\infty(X)$ in $\ell^2(X,m)$. Hence, $e^{-tL}=je^{-tL}$ is a composition of a 
continuous map from $\ell^\infty(X)$ to $\ell^2(X,m)$ and a continuous map from $\ell^2(X,m)$ to $\ell^\infty(X)$. Thus, $e^{-tL}$ is Hilbert-Schmidt by the Grothendieck factorization
principle, see  \cite{Sto94}. Since $t>0$ was arbitrary, we infer that $e^{-tL}=e^{-\frac12 tL}e^{-\frac12 tL}$ is trace class as a product of two Hilbert-Schmidt operators.

The ``in particular''-part is clear.
\end{proof}

\begin{proof}[Proof of Theorem~\ref{spectral_bound_dist_gen}]
 
 We make use of the Courant-Fischer-Weyl min-max principle and the $\ell^\infty_{\alpha,o}$-Sobolev inequality with constant $S:=S_{\infty,\alpha,o}(Q)$ and calculate 
\begin{align*}
\lambda_{k+1}&=\sup_{f_1,\ldots,f_k\in\ell^2(X,m)}\inf\left\{\frac{Q(f)}{\|f\|_2^2}\colon f\in D(Q), f\perp \operatorname{span}\{f_1, \ldots, f_k\}\right\}\\
&\geq \inf\left\{\frac{Q(f)}{\|f\|_2^2}\colon f\in D(Q), f(x_1)=\ldots=f(x_k)=0\right\}\\
&= \inf\left\{\frac{Q(f)+\alpha|f(x_1)|^2}{\|f\|_2^2}\colon f\in D(Q), f(x_1)=\ldots=f(x_k)=0\right\}\\
&= \inf\left\{\frac{Q(f)+\alpha|f(o)|^2}{\|f\|_2^2}\colon f\in D(Q), f(o)=f(x_2)=\ldots=f(x_k)=0\right\}\\
&\geq \inf\left\{\frac{\|f\|_\infty^2}{S\|f\|_2^2}\colon f\in D(Q), f(x_1)=\ldots=f(x_k)=0\right\}.
\end{align*}
Finally, note that $\|f\|_2^2\leq m(X\setminus\{x_1,\ldots, x_k\})\|f\|_\infty^2$ if $f(x_1)=\ldots=f(x_k)=0$. This concludes the proof.
\end{proof}

\begin{proof}[Proof of Theorem~\ref{spectral_bound_lin}]
By Proposition~\ref{spectrum_discrete} the spectrum of $L$ is purely discrete. 
Let $(\phi_j)$ be an orthonormal basis of $\ell^2(X,m)$ consisting of eigenfunctions of $L$ corresponding to the eigenvalues $(\lambda_j)$. 
For $k\in\mathbb{N}$ and $\beta_1,\dots,\beta_k\in\IC$, let $$ \varphi=\sum_{j=1}^k \beta_j \phi_j. $$
Denote $S=S_{\infty,\alpha,o}(Q)$. By the $\ell_{\alpha,o}^\infty$-Sobolev inequality we have
\begin{align*}
\norm{\varphi}_\infty^2&\leq S(Q(\varphi)+\alpha|\varphi(o)|^2)\\
&=S( \langle L\varphi,\varphi\rangle+\alpha|\varphi(o)|^2)\\
&\leq S\left( \langle L\varphi,\varphi\rangle+\frac{\alpha\|\varphi\|_2^2}{m(o)}\right)\\
&=S\sum_{j=1}^k \left(\lambda_j+\frac{\alpha}{m(o)}\right)\abs{\beta_j}^2\\
&\leq S\left(\lambda_k+\frac{\alpha}{m(o)}\right)\norm{\varphi}_2^2.
\end{align*}
Now, fix $x\in X$ and let $\beta_j=\overline{\phi_j(x)}$. Then,
\begin{align*}
\sum_{j=1}^k\abs{\phi_j(x)}^2 =\varphi(x) \leq  \norm{\varphi}_\infty &\leq S^{1/2}\left(\lambda_k+\frac{\alpha}{m(o)}\right)^\frac12 \norm{\varphi}_2\\
&=S^{1/2}\left(\lambda_k+\frac{\alpha}{m(o)}\right)^\frac12\left(\sum_{j=1}^k \abs{\phi_j(x)}^2\right)^{1/2}.
\end{align*}
Thus, 
\begin{align*}
\sum_{j=1}^k \abs{\phi_j(x)}^2 \leq S \left(\lambda_k+\frac{\alpha}{m(o)}\right) 
\end{align*}
for every $x\in X$. Summing over all $x\in X$, we obtain
\begin{align*}
k=\sum_{j=1}^k \norm{\phi_j}_2^2=\sum_{x\in X}m(x)\sum_{j=1}^k \abs{\phi_j(x)}^2 \leq S \left(\lambda_k+\frac{\alpha}{m(o)}\right) m(X).
\end{align*}
Therefore,
\begin{flalign*}
&&\lambda_k\geq \frac{k}{S m(X)}-\frac{\alpha}{m(o)}.&&\qedhere
\end{flalign*}
\end{proof}

\section{\texorpdfstring{Eigenvalue estimates for $p>2$}{Eigenvalue estimates for p>2}}\label{sec:p>2}
In this section, we use the approach of Cheng and Li \cite{CL81} to estimate the eigenvalues of generators of Dirichlet forms that
satisfy $\ell_{\alpha,o}^p$-Sobolev inequalities
for $p\in(2,\infty]$. In particular, we prove Theorem~\ref{thm:pSobolev}.

Let $Q$ be a Dirichlet form on  $\ell^2(X,m)$, $L$ the infinitesimal generator, and  $e^{-tL}$ the semigroup associated to the Dirichlet form $Q.$ 
The \emph{heat kernel} of $Q$ is given by 
$$p_t(x,y)=\frac{1}{m(y)}(e^{-tL}1_y)(x),\quad x,y\in X.$$
The semigroup $e^{-tL}$ can be expressed in terms of the heat kernel as
\[ e^{-tL}f(x) = \sum_{y \in X} p_t(x,y) f(y) m(y) \]
for all $f \in \ell^2(X,m)$.

We first show that $(S_\infty)$ implies \eqref{def:sb} for any $p>2.$
\begin{proposition}\label{prop:finite}Let $(X,m)$ be a discrete measure space with finite   measure $ m(X)<\infty $. 
Suppose that $Q$ satisfies $(S_\infty)$ with $\ell^\infty$-Sobolev constant $S_\infty(Q)$, then
it satisfies $(S_p)$ for any $p>2$ with $\ell^p$-Sobolev constant $$S_p(Q)\leq S_\infty(Q)m(X)^{\frac2p}.$$
\end{proposition}
\begin{proof} For any $p>2,$ $f\in D(Q),$ finiteness of the measure and  $(S_\infty)$ yields
$$\|f\|_p^2\leq m(X)^{\frac2p}\|f\|_{\infty}^2\leq S_\infty(Q)m(X)^{\frac2p}Q(f)^{\frac12}.$$ This proves the proposition.
\end{proof}

Adapting the strategy from \cite{CL81}, we derive from the Sobolev inequality \eqref{def:sb} an on-diagonal heat kernel estimate that yields 
a lower bound on the eigenvalues.

\begin{proof}[Proof of Theorem~\ref{thm:pSobolev}] 
First let $2<p<\infty$.
 For $x\in X$, we write 
$$f\colon (0,\infty)\lra (0,\infty),\quad f(t)=p_{2t}(x,x)=\frac{1}{m(x)}(e^{-2tL}1_x)(x).$$
We calculate the derivative of $f$.
We have
$$f'(t)=-2\frac{1}{m(x)}(Le^{-2tL}1_x)(x),$$
since $L$ is the generator of $e^{-tL}$. Moreover, from $e^{-tL}1_x\in D(L)$, we deduce $Le^{-2tL}1_x=e^{-tL}Le^{-tL}1_x$ and, thus,
$$f'(t)=-2\frac{1}{m(x)}(e^{-tL}Le^{-tL}1_x)(x)=-2\frac{1}{m(x)^2}\langle e^{-tL}Le^{-tL}1_x,1_x\rangle.$$
 Using the self-adjointness of the semigroup, we infer
$$f'(t)= -2\frac{1}{m(x)^2}\langle Le^{-tL}1_x, e^{-tL}1_x\rangle=-2\sum_{y\in X}p_{t}(x,y)Lp_t(x,\cdot)(y)m(y)$$ for every $t>0$.

Thus, by Green's formula and the $\ell^p$-Sobolev inequality \eqref{def:sb},
\begin{align*}
f'(t)
&=-2Q(p_{t}(x,\cdot))\leq -2 (S_p(Q))^{-1} \|p_{t}(x,\cdot)\|_p^2\\&\leq-2 (S_p(Q))^{-1} \left(\sum_{y\in X}p_{t}^2(x,y)m(y) \right)^{2-\frac2p}=-2(S_p(Q))^{-1} f(t)^{2-\frac{2}{p}},
\end{align*}
 where we have used the H\"older inequality and the fact $$\sum_{y\in X}p_{t}(x,y)m(y)\leq 1$$ in the last inequality. For $p>2,$ this implies that 
$$\left(f(t)^{\frac2p-1}\right)'\geq 2(S_p(Q))^{-1}\left(1-\frac2p\right).$$ Integrating the variable $t$ from $0$ to $t$ on both sides, we have
$$f(t)^{\frac2p-1}\geq f(0)^{\frac2p-1}+2(S_p(Q))^{-1}\left(1-\frac2p\right)t\geq 2(S_p(Q))^{-1}\left(1-\frac2p\right)t.$$ This yields
$$p_{2t}(x,x)=f(t)\leq C_1 t^{-\frac{p}{p-2}},$$ where $C_1=\left(\frac{S_p(Q)}{2\left(1-\frac2p\right)}\right)^{\frac{p}{p-2}}.$ 
Summing over $x\in X$ with weights $m$ on both sides of the above inequality, we get
$$\sum_{y\in X}p_{2t}(x,x) m(x)\leq C_1m(X)t^{-\frac{p}{p-2}}.$$ 
In particular, this gives $$\sum_{x\in X}\sum_{y\in X}p_t(x,y)^2 m(y)m(x)=\sum_{x\in X} p_{2t}(x,x)m(x)<\infty$$ and, hence, 
$p_t(x,y)$ is a Hilbert-Schmidt kernel for every $t>0$. Thus, $(e^{-tL}$ are Hilbert-Schmidt operators and, therefore, $e^{-tL}=e^{-\frac12 tL}e^{-\frac12 tL}$ is trace class as a product of two Hilbert-Schmidt operators. Hence, $L$ has purely discrete spectrum.

We denote $$\overline{\lambda}_k:=\frac1k\sum_{i=1}^k\lambda_i.$$
By the trace formula for Hilbert-Schmidt operators and by the convexity of the function $x\mapsto e^{-x}$, we calculate 
$$ke^{-2\overline{\lambda}_k t}\leq \sum_{i=1}^{k}e^{-2\lambda_i t}\leq \sum_{i\geq 1}e^{-2\lambda_i t}=\sum_{x\in X}p_{2t}(x,x) m(x).$$ 
This implies that for any $t\geq 0,$
$$k\leq C_1 m(X)e^{2\overline{\lambda}_k t}t^{-\frac{p}{p-2}}.$$ Minimizing the right hand side for $t\geq 0,$ i.e., setting 
$t=\frac{p}{2\overline{\lambda}_k(p-2)},$ we get 
$$\overline{\lambda}_k\geq \frac{1}{eS_p(Q)}\left(\frac{k}{m(X)}\right)^{1-\frac{2}{p}}.$$ This proves the first assertion. 
The second assertion follows from $\lambda_k\geq \overline{\lambda}_k.$

Finally, we discuss the case $p=\infty$.
 By Proposition~\ref{prop:finite}, the graph $(b,c)$ over $ (X,m) $ satisfies \eqref{def:sb} for every $p>2$ with $\ell^p$-Sobolev constant 
$$S_p(Q)\leq S_\infty(Q)m(X)^{\frac2p}.$$ By passing to the limit, $p\to\infty,$ in the estimate we have proven above the result follows.
\end{proof}

\section{Application to graphs}\label{sec:ex}
In this section, we apply eigenvalue estimates to uniformly transient graphs and canonically compactifiable graphs respectively.
For this, we recall some basics concerning Dirichlet forms on graphs. 
The articles \cite{KL10,KL12} can serve as a reference.

Let $X$ be an at most countable set. A \emph{graph over $X$} is a pair $(b,c)$ consisting of a symmetric function $b\colon X\times X\lra[0,\infty)$ that vanishes on the diagonal and satisfies
$$\sum_y b(x,y)<\infty$$ for all $x\in X$,
and a function $c\colon X\lra [0,\infty)$. We call $b$  \emph{edge weight} and $c$ the \emph{killing term}. 

Two vertices $x$, $y$ are called \emph{adjacent}, written $x\sim y$, if $b(x,y)>0$. In this way, a graph $(b,c)$ over $X$ induces a simple, undirected graph in the combinatorial sense with vertex set $X$ and edge set $E=\{(x,y)\in X\times X\mid x\sim y\}$. The graph $(b,c)$ is called \emph{connected} if for all $x,y\in X$ there exist $x_1,\dots,x_n\in X$ such that $x\sim x_1$, $x_n\sim y$ and $x_i\sim x_{i+1}$ for $1\leq i\leq n-1$.

We  equip $X$  with the discrete topology and denote by $C(X)$, $C_c(X)$  for the space of all, all finitely supported functions on $X$. A central object in the analysis on graphs is the \emph{energy form associated with $(b,c)$}. It is defined as
\begin{align*}
\tilde Q\colon C(X)\lra [0,\infty],\quad f\mapsto \frac 1 2\sum_{x,y\in X}b(x,y)\abs{f(x)-f(y)}^2+\sum_{x\in X} c(x)\abs{f(x)}^2.
\end{align*}
The map $ \tilde Q $ is a quadratic form with domain
\begin{align*}
\tilde D=\{f\in C(X)\mid \tilde Q(f)<\infty\}
\end{align*}
is called the space of \emph{functions of finite energy}. A crucial property of $\tilde Q$ is that it is \emph{Markovian}, that is,
\begin{align*}
\tilde Q(C\circ f)\leq \tilde Q(f)
\end{align*}
for all $f\in C(X)$ and all normal contractions $C\colon \IC\lra\IC$. Moreover, by Fatou's lemma, $\tilde Q$ is sequentially lower semicontinuous with respect to pointwise convergence.
For $o\in X$, define
\begin{align*}
\norm\cdot_o\colon \tilde D\lra [0,\infty),\,f\mapsto (\tilde Q(f)+\abs{f(o)}^2)^{1/2}.
\end{align*}
It is easy to see that $\norm\cdot_o$ is a norm, and if $(b,c)$ is connected, then $\norm\cdot_o$ and $\norm\cdot_{o'}$ are equivalent for any $o,o'\in X$ and we denote  $$ \tilde D_0=\overline{C_c(X)}^{\norm\cdot_o} .$$  

Let $m\colon X\lra (0,\infty)$ such that the pair $(X,m)$ is a discrete measure space.

The restriction $Q^{(N)}$ of $\tilde Q$ to $\ell^2(X,m)$ is a closed, densely defined quadratic form in $\ell^2(X,m)$, called the \emph{Dirichlet form with Neumann boundary conditions}.
 Accordingly, the restriction of $\tilde Q$ to $C_c(X)$ is closable in $\ell^2(X,m)$, and its closure $Q^{(D)}$ is called the \emph{Dirichlet form with Dirichlet boundary conditions}.
 It turns out (see \cite[Appendix A]{KLSW17}) that
$$D(Q^{(D)})=\tilde D_0\cap\ell^2(X,m).$$
We write $L^{(D)}$ and $L^{(N)}$ for the  generators of $Q^{(D)}$ and $Q^{(N)}$ (which are {positive} operators on $\ell^2$) and call them the \emph{Laplacian with Dirichlet} respectively\emph{ Neumann boundary conditions}.

For any closed quadratic form $Q$ on $\ell^2(X,m)$ we write $\norm{\cdot}_Q$ for the \emph{form norm} given by
\begin{align*}
\norm{f}_Q=(Q(f)+\norm{f}_2^2)^{1/2}.
\end{align*}
The Dirichlet form $Q^{(D)}$ is regular, that is, $D(Q^{(D)})\cap C_c(X)$ is dense both in $C_c(X)$ with respect to $\norm\cdot_\infty$ and in $D(Q^{(D)})$ with respect to $\norm\cdot_{Q^{(D)}}$. Conversely, every regular Dirichlet form on a discrete measure space $(X,m)$ is given by $Q^{(D)}$ for some graph $(b,c)$ over $X$, see \cite{KL12}.

\subsection{Uniformly Transient Graphs}\label{section:uniform}
A connected graph $(b,c)$ over a countably infinite set $X$ is called \emph{uniformly transient} if there is a constant $C>0$ such that   $$\|f\|_\infty^2\leq C\tilde{Q}(f)$$ 
holds for every $f\in\tilde D_0$. 
Uniformly transient graphs are studied in \cite{KLSW17}. 
By the definition of uniformly transient graphs and since $D(Q^{(D)})=\tilde D_0\cap \ell^2(X,m)$ holds, one immediately gets the following lemma. 
\begin{lemma}\label{Sobolev_embedding_unif}
If $(b,c)$ is a connected uniformly transient graph over $ (X,m) $, then there exists a constant $C>0$ such that for all $o\in X$ the Dirichlet form $Q^{(D)}$ satisfies a $\ell^\infty_{\alpha,o}$-Sobolev inequality, i.e.,
\begin{align*}
\norm{f}_\infty^2\leq C Q^{(D)}(f)
\end{align*}
holds for all $f\in D(Q^{(D)})$. Moreover, the constant $C$ is independent of the choice of the  measure $m$.
\end{lemma}
The previous lemma combined with Proposition~\ref{spectrum_discrete} immediately yields that on a uniformly transient graph the operator $ L^{(D)} $ associated to $Q^{(D)}$ has purely discrete spectrum. Moreover, we can 
apply Theorem~\ref{spectral_bound_lin}, which immediately gives the following corollary.
\begin{corollary}\label{Dirichlet_lower_bound_lin}
Let $(b,c)$ be a uniformly transient graph over $ (X,m) $ with finite  measure $ m(X)<\infty $. Let $L^{(D)}$ be the generator of $Q^{(D)}$. Then, $ L^{(D)} $ has purely discrete spectrum and 
for all $k\in\mathbb{N}$. 
\begin{align*}
\lambda_k\geq \frac{k}{C m(X)}
\end{align*}
where $\lambda_k$ denotes the $k$-th eigenvalue of $L^{(D)}$  counted with multiplicity.
 \end{corollary}
Moreover, the following theorem, which is taken from \cite[Theorem~7.2, 7.3]{KLSW17}, is a direct consequence of Theorem~\ref{spectral_bound_dist_gen}.

\begin{theorem}[Theorem~7.2, 7.3 in \cite{KLSW17}]\label{thm:lower_bound_unif_trans}  Let $(b,c)$ be a uniformly transient graph over $ (X,m) $ with finite  measure $ m(X)<\infty $. 
Then the spectrum of Dirichlet Laplacian $L^{(D)}$ of the graph is discrete.  Moreover, if $\{x_n\}$ is an enumeration of $X,$ then for any $k\geq 1,$ 
\begin{equation*}
\label{eq:keller_original}\lambda_{k}\geq \frac{1}{S_\infty(Q)m(X\setminus\{x_1,\cdots,x_{k-1}\})},
\end{equation*}
where $\lambda_k$ denotes the $k$-th eigenvalue of $L^{(D)}$   counted with multiplicity.
\end{theorem}

Using Theorem~\ref{thm:lower_bound_unif_trans}  one can show that  for a given graph $(b,0)$ and a given total mass of $m$, 
there is no upper bound on the growth of the eigenvalues of the Dirichlet Laplacian.
\begin{proposition}\label{no_upper_bound}
Let $(b,0)$ be a connected uniformly transient graph over $X$. For every sequence $(a_k)$ of positive numbers there exists a probability measure $m$ on $X$ such that the eigenvalues $\lambda_1\leq \lambda_2\leq \dots$ of $L^{(D)}$ satisfy
\begin{align*}
\frac{\lambda_k}{a_k}\to \infty.
\end{align*}
\end{proposition}
\begin{proof}
We may assume without loss of generality that $(a_k)$ is increasing and $a_k\geq k$ for all $k\in\mathbb{N}$. Let $(x_j)$ be an enumeration of $X$ and let 
\begin{align*}
m(x_j)=a_1^2\left(\frac 1 {a_j^2}-\frac  1{a_{j+1}^2}\right).
\end{align*}
Then $m$ is a probability measure (telescope sum) and
\begin{align*}
m(X\setminus\{x_1,\dots,x_k\})=a_1^2\sum_{j=k+1}^\infty \left(\frac 1 {a_j^2}-\frac 1{a_{j+1}^2}\right)=\frac{a_1^2}{a_{k+1}^2}.
\end{align*}
It follows from Theorem~\ref{thm:lower_bound_unif_trans} that
\begin{align*}
\frac{\lambda_{k+1}}{a_{k+1}}\geq \frac {a_{k+1}^2}{C a_1^2 a_{k+1}}=\frac{a_{k+1}}{C a_1^2}\to \infty
\end{align*}
as $k\to\infty$.
\end{proof}
By the previous lemma, there is no general upper bound for the eigenvalue growth of Dirichlet forms that satisfy a $\ell^\infty$-Sobolev inequality. However, the next lemma  shows that 
the linear growth is, in some sense, the best lower bound one can achieve.
\begin{proposition}\label{prop:sharpexample}
Let $(b,0)$ be a uniformly transient graph such that the map $\operatorname{deg}: X\to [0,\infty)$, $\operatorname{deg}(x)=\sum_{y\in X} b(x,y)$ is bounded. Then, for every $\varepsilon>0$ there is a finite measure $m$ on $X$ such that the eigenvalues  $\lambda_1\leq\lambda_2\leq \ldots$ of $L^{(D)}$ satisfy 
\begin{align*}
\frac{\lambda_k}{k^{1+\varepsilon}}\to 0.
\end{align*} 
\end{proposition}
  \begin{proof}
  Let $D>0$ be an upper bound on $\operatorname{deg}$. Let $x_1,x_2,\ldots$ be an enumeration of $X$ and let $\varepsilon>0$ be arbitrary.
  We define $ m $ to be the finite measure given via $m(x_n)=n^{-(1+\frac\varepsilon2)}$.  An easy calculation shows
$$Q^{(D)}(f)\leq 2\langle \frac{1}{m}\mathrm{deg} \,f,f\rangle\leq \langle\frac{2D}{m} f,f\rangle$$
for every $f\in D(Q^{(D)})$. 
By the min-max principle we infer $ \lambda_{k}\leq \mu_{k} $, where $ \mu_{k} $ denote the $ k $-th eigenvalues of the multiplication operator by the function $ 2D/m $ on $ \ell^{2}(X,m) $. Moreover, by the definition of the measure, we infer  $$ \mu_{k}=\frac{2D}{m(x_k)}={2D}{k^{1+\frac\varepsilon2}} $$ which concludes the proof.
  \end{proof}
\begin{example}
Uniformly transient graphs $(b,0)$ with bounded $\operatorname{deg}$ are for example the regular binary tree with standard weights, c.f. \cite[Example~2.5.]{KLSW17}, and $\mathbb{Z}^d$, $d\geq 3$ with $b(x,y)=1$ if $x$ and $y$ have Euclidean distance one, and $b(x,y)=0$ otherwise, c.f. \cite[Section 6]{KLSW17}. Moreover, there it is shown that $\mathbb{Z}^d$, $d\geq 3$, is a canonically compactifiable graph, (see the next section for the definition).  
\end{example}

\subsection{Canonically Compactifiable Graphs}\label{section:canon}

An infinite connected graph $(b,c)$ over $(X,m)$ is called \emph{canonically compactifiable} if
$$ D(Q^{(N)})\subseteq \ell^\infty(X) .$$
Canonically compactifiable graphs were first defined in \cite{GHKLW15} under the slightly stronger hypothesis $ \widetilde{D}\subseteq \ell^{\infty}(X) $, which does not include the measure. There it was shown that that Laplacians on these graphs share indeed many properties with Laplacians on pre-compact manifolds with nice boundary.  In   \cite{KS17}, non-linear equations were studied on these graphs using the condition including the measure as we have introduced it above.

For Dirichlet forms $Q_1,Q_2$ we write $Q_1\geq Q_2$ if $D(Q_1)\subseteq D(Q_2)$ and $Q_1(u)\geq Q_2(u)$ for every $u\in D(Q_1)$. 
From the definitions it is clear that $Q^{(D)}\geq Q^{(N)}$.
In this section we want to study the growth of eigenvalues of the operators associated with  Dirichlet forms $Q$ satisfying  $Q^{(D)}\geq Q\geq Q^{(N)}$.

Note, that in \cite{KLSW17} it was shown that canonically compactifiable graphs are uniformly transient. Hence, Corollary~\ref{Dirichlet_lower_bound_lin} and Theorem~\ref{thm:lower_bound_unif_trans} can be applied to $Q^{(D)}$. Therefore, we immediately have lower bounds on the growth of the eigenvalues of the Dirichlet Laplacian. However, on canonically compactifiable graphs every Dirichlet form between 
$Q^{(D)}$ and $Q^{(N)}$ satisfies a $\ell^\infty_{\alpha,o}$-Sobolev inequality. This is the content of the next lemma.
\begin{lemma}\label{Canon_Sobolev_embedding}
Let $(b,c)$ be a canonically compactifiable graph over $ (X,m) $ with finite  measure $ m(X)<\infty $ and let $Q$ be a Dirichlet form such that $Q^{(D)}\geq Q\geq Q^{(N)}$.  Then, for every $o\in X$ there exists a constant $C>0$ such that $Q$ satisfies the $\ell^\infty_{\alpha,o}$-Sobolev inequality
\begin{align*}
\norm{f}_\infty^2\leq C (Q(f)+|f(o)|^2)
\end{align*}
for all $f\in  D(Q)$.
\end{lemma}
\begin{proof}
Since we have $Q(f)\geq Q^{(N)}(f)$ for every $f\in D(Q)$, we only have to prove the inequality for $Q^{(N)}$. This, however, is shown in
{\cite[Lemma~4.2]{GHKLW15}}.
\end{proof}
The previous lemma combined with Proposition~\ref{spectrum_discrete} immediately shows that on a uniformly transient graph the form $Q^{(D)}$ has compact resolvent. Moreover, we can apply
Theorem~\ref{spectral_bound_lin} and Theorem~\ref{spectral_bound_dist_gen}, which immediately give the following corollaries.
\begin{corollary}
Let $(b,c)$ be a canonically compactifiable graph over $ (X,m) $ with finite  measure $ m(X)<\infty $. Let $L$ be the generator of a Dirichlet form $Q$ with $Q\geq Q^{(N)}$. Then for every $o\in X$ there is a $C>0$ such that 
for all $k\in\mathbb{N}$
\begin{align*}
\lambda_k\geq \frac{k}{C m(X)}-\frac{1}{m(o)}
\end{align*}
where $\lambda_k$ denotes the $k$-th eigenvalue of $L$   counted with multiplicity. 
\end{corollary}
\begin{corollary}Let $(b,c)$ be a canonically compactifiable graph over $ (X,m) $ with finite  measure $ m(X)<\infty $. Let $L$ be the generator of a Dirichlet form $Q$ with $Q\geq Q^{(N)}$ and let $(x_k)_{k\in\mathbb{N}}$ be an enumeration of $X$. Then, there is $C>0$ such that  
 for all $k\in\mathbb{N}$
\begin{align*}
\lambda_{k+1}\geq \frac 1{C m(X\setminus\{x_1,\dots,x_k\})}
\end{align*}
where $\lambda_k$ denotes the $k$-th eigenvalue of $L$   counted with multiplicity.
\end{corollary}
\begin{proof}
This follows by Theorem~\ref{spectral_bound_dist_gen} and since $Q$ satisfies a $\ell_{1,x_1}^\infty$-Sobolev inequality by Lemma~\ref{Canon_Sobolev_embedding}.
\end{proof}

\bibliography{mybib}{}
\bibliographystyle{alpha}

\end{document}